\documentclass[a4paper,dvipdfmx,12pt]{article}
\usepackage{amsmath,amssymb,amsfonts,mathrsfs,geometry,color}
\usepackage{geometry}
\usepackage{enumitem}
\usepackage{overpic}

\usepackage{amsthm}
\theoremstyle{definition}
\newtheorem{Thm}{Theorem}[section]
\newtheorem*{Thm*}{Theorem}

\newtheorem*{Def*}{Definition}
\newtheorem{Prop}[Thm]{Proposition}
\newtheorem*{Prop*}{Proposition}
\newtheorem{Lem}[Thm]{Lemma}
\newtheorem*{Lem*}{Lemma}

\newtheorem*{Cor*}{Corollary}
\newtheorem{Exam}[Thm]{Example}
\newtheorem*{Exam*}{Example}
\newtheorem*{Notation*}{Notation}
\newtheorem{Rem}[Thm]{Remark}
\newtheorem*{Rem*}{Remark}

\numberwithin{equation}{section} %数式番号を節ごとにする。

\newcommand{\N}{\mathbb{N}}

\newcommand{\R}{\mathbb{R}}
\newcommand{\C}{\mathbb{C}}

\DeclareMathOperator{\diag}{diag}% 対角行列のため
\DeclareMathOperator{\Tr}{Tr}% トレースのため

\newcommand\mm{m}

\newcommand\hA{\widehat{A}}
\newcommand\hB{\widehat{B}}
\newcommand\hU{\widehat{U}}
\newcommand\tA{\widetilde{A}}
\newcommand\tB{\widetilde{B}}

\newcommand{\subcases}[1]{\vspace{2mm}\noindent\underline{\bf #1}}

\title{Fluctuations of eigenvalues of a polynomial on Haar unitary and finite rank matrices}
\author{Beno\^{i}t Collins, Katsunori Fujie, Takahiro Hasebe, \\ Felix Leid and Noriyoshi Sakuma}
\date{\today}

\begin{document}

\maketitle

\begin{abstract}
This paper calculates the fluctuations of eigenvalues of polynomials on large Haar unitaries cut by finite rank deterministic matrices. When the eigenvalues are all simple, we can give a complete algorithm for computing the fluctuations. When multiple eigenvalues are involved, we present several examples suggesting that a general algorithm would be much more complex. 

\end{abstract}

%\tableofcontents

%%%%%%%%%%%%%%%%%%%%%%%%%%%%%%%%%%%%%%%%%%%%%%%%%%%
%%%%%%%%%%%%%%%%%%%%%%%%%%%%%%%%%%%%%%%%%%%%%%%%%%%
%%%%%%%%%%%%%%%%%%%%%%%%%%%%%%%%%%%%%%%%%%%%%%%%%%%
\section{Introduction}

In random matrix theory, the behavior of large eigenvalues offers valuable insights, particularly regarding their positions and fluctuations. 
For example, significant statistical results have been derived from studies like \cite{Johnstone2001,Paul2007}. 
A particularly intriguing phenomenon for researchers is the BBP phase transition, as discussed in works like \cite{BBP2005} and \cite{Peche2006}. 
Analyzing this transition poses several challenges, often requiring advanced techniques such as the moment method and complex analysis. 
While these are standard in the study of random matrix theory, many cases demand intensive calculations for each model \cite{FerPe2007, CapFer2009, Benaych2011}. Conversely, non-commutative probability techniques, such as freeness, second order freeness, and infinitesimal freenesses, have proven advantageous for systematically analyzing models based on polynomials of several random matrices \cite{MingoSpeicher, Shlyakhtenko2018}.  %Second order freeness gave a systematic way to obtain Gaussian fluctuations for global quantities of polynomials of typical random matrices \cite[Chapter 5]{MingoSpeicher}. 

Our previous works \cite{CHS18} and \cite{CLS} have employed free probability and the moment method to systematically analyze the large eigenvalues of models constructed from polynomials of multiple random matrices, including those of finite rank. 
Notably, the concept of cyclically monotone independence has been instrumental in computing their moments and revealing their underlying phenomema. This concept was further developed in  \cite{Cebron2023} where they gave a deep explanation of outlier problem with the moment method. %type B freeness, conditional freeness and cyclic monotone independence. 
Cyclically monotone independence originates from infinitesimal freeness and is connected to recent research on type B freeness, conditional freeness, and cyclic boolean independence \cite{Arizmendi2021,Arizmendi2023,Cebron2023,FujieHasebe,ArizmendiCebron,CebronGilliers2024}.

%In the paper by the first, third and last authors, we cannot inculde Baik, Ben-Arous, P{\'e}ch{\'e} phase transition, i.e. the additive finite rank perturbed GUE model and related works. They use type B free probability theory and solved our problem.

Now it is the turn to consider their fluctuation. In free probability, second order freeness gave a systematic way to obtain Gaussian fluctuations for global quantities of polynomials of typical random matrices \cite[Chapter 5]{MingoSpeicher}. In this paper, we consider random matrices with only discrete eigenvalues in the large $N$ limit as in \cite{CHS18} and present a method for computing the fluctuations of eigenvalues, which provides a deeper understanding of outlier problem.
More concretely, the present paper analyzes limiting eigenvalues and their fluctuations  of the $N\times N$ random matrix
\begin{equation}\label{eq:model}
P(A_1U^*, A_2 U^*, \dots, A_k U^*,  U B_1, U B_2, \dots, U B_\ell) 
\end{equation}
in the large $N$ limit, where $P$ is a polynomial in $k+\ell$ noncommuting indeterminates without a constant term, $U\equiv U^{(N)} =(u_{ij})_{i,j\in[N]}\equiv (u_{ij}^{(N)})_{i,j\in[N]}$ is a Haar unitary matrix, 
\begin{align}\label{eq:model_hat}
 A_i = 
\begin{pmatrix}  \hA_i  & 0  \\
   0 & 0 \\ 
       \end{pmatrix}  \in M_N(\C) \qquad \text{and}  \qquad 
  B_i =\begin{pmatrix}  \hB_i  & 0  \\
   0 & 0 \\ 
       \end{pmatrix} \in M_N(\C)
\end{align}
with $\hA_i, \hB_i \in M_r(\C)$.  The number $r \in \N$ is fixed, and $N$ is always assumed to satisfy sufficiently large $N$.  We basically exhibit all possible patterns of the model \eqref{eq:model} providing methods for calculating limiting eigenvalues and fluctuations, and show that when multiple eigenvalues appear the number of patterns can be very huge.   

%It depends of the polynomial. 

This model is related to the model in \cite{CHS18} in the sense that both involve Haar unitaries and the limiting eigenvalues are discrete in the large $N$ limit. Although the paper \cite{CHS18} focused only on almost sure convergence of eigenvalues, the present paper is mainly concerned with fluctuations of the eigenvalues.  
%\textcolor{red}{Please explain the relation to Collins, Leid, Sakuma.}

Let us explain roughly the idea behind the construction proposed in \cite{CLS}.
It relies on the intuition that for any vector subspace $V$ of dimensions $r$ of $\C^N$, if we consider the image $U\cdot V$ of $V$ under the Haar unitary  $U$, then $V$ and $U\cdot V$ are almost orthogonal in the sense that the inner product between any normed
vector of $V$ and a normed vector of $U\cdot V$ is uniformly close to zero with high probability. 
This intuition can be lifted at the level of matrices as follows: 
for the Hilbert-Schmidt norm, any  $\tilde A_i$ of norm one with the same domain and codomain as a  matrix $A_i$ of norm one satisfies the property that $U\tilde A_i$ and $\tilde A_i U^*$ are
almost orthogonal to $A_i$ in a uniform sense. A perfect orthogonality (which, in a sense, occurs when $N\to\infty$) gives naturally rise to the construction of \cite{CLS}. In a sense, $U\tilde A_i$ is obtained from $\tilde A_i$ by making a ``block row operation'' and sending 
$\tilde A_i$ to its almost orthogonal self, whereas $\tilde A_i U^*$ is obtained from $\tilde A_i$ by making a ``block column operation.''

For finite $N$, the goal of this paper is to try to view $P(A_1U^*, \dots, A_k U^*,  U B_1, \dots, U B_\ell) $
as an $o(N^{-1})$ perturbation of the model of \cite{CLS}, and deduce the fluctuations of the random matrix model from those of the limit model with perturbative methods.
In this paper, we obtained the following results on eigenvalue fluctuations based on the above idea.

%The case of $k=l=1$ is essential, as can be seen immediately from the algorithm that we will describe in Section \ref{sec:algorithm}.

\begin{Thm} \label{thm:converge.a.s.} 
The matrix $P(A_1U^*, \dots, A_k U^*,  U B_1, \dots, U B_\ell) $ has $N-2r$ zero eigenvalues, called the ``trivial eigenvalues''. The other eigenvalues, called the ``nontrivial eigenvalues'' (although zeros may be included), converge almost surely to deterministic numbers as $N\to \infty$.  See Subsection \ref{sec:algorithm} for an algorithm for computing these limits. 
\end{Thm}
 Here, the term  ``trivial eigenvalues'' means that they are always identical to zero independently of the polynomials and the $\widehat A_i, \widehat B_j$'s.  Nontrivial eigenvalues may or may not be zero depending on a model. 

Let $\{\mu_{i}^{(N)}\}_{i=1}^{2r}$ be the nontrivial eigenvalues of $P(A_1U^*, \dots, A_k U^*,  U B_1, \dots, U B_\ell) $ and $\{\mu_{i}\}_{i=1}^{2r}$ denote their limits of the eigenvalues as described in Theorem \ref{thm:converge.a.s.}. 
\begin{Thm} \label{thm:fluctuations}
  In addition, assume that all these values $\{\mu_{i}\}_{i=1}^{2r}$ appear without multiplicity. Then, for every $i\in[2r]$, the number
  \begin{equation}\label{eq:exponent1}
  \kappa_i := \sup\{\kappa \in \R\mid N^{\frac{\kappa}{2}}(\mu_{i}^{(N)}-\mu_{i}) \text{~converges in law to a $\C$-valued random variable} \}
  \end{equation}
belongs to the set $\N \cup\{\infty\}$. Here, $\kappa_i=\infty$ means that $N^{\frac{\kappa}{2}}(\mu_{i}^{(N)}-\mu_{i})$ converges in law to 0 for all $\kappa\in\R$, which  occurs only when $\mu_{i}^{(N)}=\mu_{i}$ a.s.\ for all sufficiently large $N \in\N$.   

Moreover, let $I:=\{i \in [2r]\mid \kappa_i <\infty\}$.  Then the random vectors $(N^{\frac{\kappa_i}{2}}(\mu_{i}^{(N)}-\mu_{i}) )_{i\in I}$  converge in  law to $(P_i)_{i\in I}$ as $N\to\infty$, where $P_i=P_i(x_1,x_2,\dots, x_{2r^2})$ are nonzero homogeneous polynomials of degree $\kappa_i,  i\in I$, on a standard Gaussian random vector $(x_{i})_{i \in [2r^2]}$ on $\R^{2r^2}$. 
%  In any case, the fluctuation $N^{\kappa}(\mu_{i}^{(N)} - \mu_{i})$ converges to a finite mixture of normal distributions and exponential distributions.
\end{Thm}

%%%%%%%   
%Comments by TH: in case $N(\mu_{i}^{(N)}-\mu_{i})$, I am not sure if the limit distribution is mixture of exponentials. for example, there might be a term e.g.\ $\Re(u_{12}) \Re (u_{21})$, then this will be a product of independent normal r.v.}
%%%%%%%%%

The fluctuation limits appearing here can be obtained in principle by specific calculations. 
We will provide calculations for the two models $UA + AU^*$ and $P(A, UBU^{\ast})$. Note that the latter model is a special case of \eqref{eq:model} because $A$ and $U BU^*$ can be expressed e.g.\ as
\[
A= A U^* U \begin{pmatrix}  1_r  & 0  \\
   0 & 0 \\ 
       \end{pmatrix} 
  \qquad \text{and} \qquad 
U BU^*  = U B \begin{pmatrix}  1_r  & 0  \\
   0 & 0 \\ 
       \end{pmatrix}  U^*. 
\]
The fluctuations for $UA + AU^*$ and $P(A, UBU^{\ast})$ are normal distributions and mixtures of exponential distributions, respectively. Remarkably, fluctuations of eigenvalues of $P(A, UBU^{\ast})$ for generic polynomials $P$ can be explicitly calculated in the following way. 
\begin{Thm} \label{exam:A-UBU*}
Let $P(x,y)$ be a polynomial over $\C$ in noncommuting elements $x$ and $y$ such that $P(0,0)=0$. Let $P_1(x) := P(x,0)$ and $Q_1(y) := P(0,y)$.  
Let  
\begin{align*}
  A &= \diag(\alpha_1,\alpha_2, \dots, \alpha_r,0,0,\dots,0),\\
  B &= \diag(\beta_1,\beta_2, \dots, \beta_s, 0, 0,\dots, 0)
\end{align*}
with $\alpha_i,\beta_j \in \C\setminus\{0\}$ for all $i \in [r], j \in [s]$.  Then the $r+s$ nontrivial eigenvalues (the meaning will be made clear in the proof) of the random matrix
\begin{equation}\label{eq:special}
P(A, U B U^*)  
\end{equation}
converge to $\{P_1(\alpha_i)\}_{i=1}^r$ and $\{Q_1(\beta_j)\}_{j=1}^s$ a.s.
If these $r+s$ limiting values are all distinct, then the nontrivial eigenvalues of \eqref{eq:special} are of the forms
\begin{align*}
&P_1(\alpha_i) + \frac1{N}\sum_{j \in [s]} p_{i,j} \left|\sqrt{N}u_{i,j}^{(N)}\right| ^2 + \frac{\xi_i^{(N)}}{N^{\frac3{2}}}, \quad i \in [r] \quad \text{and} \\ 
&Q_1(\beta_j) + \frac1{N}\sum_{i \in [r]} q_{i,j} \left|\sqrt{N}u_{i,j}^{(N)}\right| ^2+ \frac{\zeta_j^{(N)}}{N^{\frac{3}{2}}},  \quad j \in [s], 
\end{align*}
where $p_{i,j}, q_{i,j}$ are explicit complex constants (shown in the proof, see  \eqref{eq:mixture_exponential1} and  \eqref{eq:mixture_exponential2}) and $\xi_i^{(N)}, \zeta_j^{(N)}$ denote random variables that converge in law to $\C$-valued random variables. The random variables $\{\big|\sqrt{N}u_{ij}^{(N)}\big|^2 \mid i\in [r], j\in[s]\}$ converge in law to standard exponential iid random variables. See Figures \ref{fig:H1} and \ref{fig:H2} for simulations. 
% times mixtures of independent exponential random variables.
\end{Thm}

\begin{Rem}
Originally the matrices $\widehat A_i$ and $\widehat B_i$ in \eqref{eq:model_hat} were assumed to have a common size $r$ and instead were allowed to have zero eigenvalues. For the model $P(A, U B U^*)$ above, however,  when $\widehat A$ or $\widehat B$ (the first $r \times r$ corners of $A$ and $B$) contains zero eigenvalues, the limiting nontrivial eigenvalues of $P(A, U B U^*)$ easily have multiple zero eigenvalues, which violates our assumption of simplicity. Therefore, we assume  in Theorem \ref{exam:A-UBU*} that $\widehat A$ and $\widehat B$ have \emph{only nonzero} eigenvalues and, instead, they are allowed to have different sizes, denoted $r$ and $s$ respectively. %we take a slightly modified setting to avoid the appearance of zeros in the nontrivial eigenvalues. 
Then one sees that the matrix $P(A,UBU^*)$ has $r+s$ ``nontrivial eigenvalues''. 
\end{Rem}

When multiple eigenvalues appear in the limit, the situation is more complex than for models that have only simple eigenvalues.  We will study some typical phenomena through the specific model $A + U B U^*$. Striking features include: 
\begin{itemize}
\item fluctuations of a multiple eigenvalue may have different orders, see Example  \ref{exa:multipleB};  
\item fluctuations can be non-polynomial functions of standard Gaussian random vectors in contrast to the case of no multiplicities, see Examples \ref{exa:multipleA}, \ref{exa:multipleB}, \ref{exa:multipleC}, cf.\  Theorem \ref{thm:fluctuations}.     
\end{itemize}

This paper is organized as follows. 
In Section 2, we will present essential lemmas for obtaining fluctuations.
In Section 3, we prove the main theorems and provide the aforementioned examples $UA + AU^*$ and $P(A, UBU^{\ast})$.
In Section 4, we will examine the model $A + U B U^*$ that has eigenvalues with multiplicities.

\section{Technical tools}
Calculations of the fluctuations of eigenvalues are based on the following two facts. Let $\widehat{U} \equiv \widehat{U}^{(N)} = (u_{ij})_{i,j \in [r]} \equiv (u_{ij}^{(N)})_{i,j \in [r]}$  be the truncation of $U$. 

\begin{Lem}[{Theorem 4.2.1 and Proposition 4.4.1 in \cite{C03}}] \label{lem:TruncatedUnitary}
For $N \ge 2r$, $\widehat{U}^{(N)}$ has the probability density function
\[
c_{N,r} \det(1_r - AA^{*})^{N-2r} 1_{\|A\| \le 1}dA,
\]
where $c_{N,r}$ is a normalization constant and $dA$ is the Lebesgue measure on $M_r(\C)$. 
In particular, as $N$ tends to infinity,
the convergence in law  
\[
\sqrt{N}\widehat{U}^{(N)}  \longrightarrow  Z
\] holds,  where $Z = (z_{ij})_{i,j\in [r]}$ is a \emph{standard complex Gaussian random matrix}, i.e., $\{\Re(z_{ij}), \Im(z_{ij}):  i,j \in [r]\}$ are i.i.d.~random variables having normal distribution with mean 0 and variance $1/2$. 
\end{Lem}

\begin{Rem}\label{rem:skorohod}
According to the Skorohod representation theorem \cite[Theorem 6.7]{Bil99}, there exist  $r\times r$ random matrices $Y, V^{(N)}, N\in\N$ on some probability space such that $Y, V^{(N)}$ have the same distributions as $Z,\widehat{U}^{(N)}$, respectively, and that $\sqrt{N}V^{(N)}$ converges to $Y$ almost surely. Some arguments below (in particular in Section \ref{sec:multiple}) can be simplified by employing $Y$ and $V^{(N)}$. 
%For notational brevity we will keep the original notation $\widehat{U}^{(N)}$ instead of the new one $V^{(N)}$. 
\end{Rem}

The previous lemma readily implies that $\widehat U^{(N)}$ itself converges to $0$ in probability.  More strongly, almost sure convergence holds. 

\begin{Lem} \label{lem:as}
 As $N\to\infty$, $\widehat{U}^{(N)}$ converges to $0$ a.s.
\end{Lem}
\begin{proof} 
It is known that $\mathbb E[|u_{ij}^{(N)}|^4] = \frac{2}{N(N+1)}$, see e.g.\ \cite[p.\ 778]{CS06}. Taking the sum over $N$ implies that $\sum_{N=1}^\infty|u_{ij}^{(N)}|^4 $ has finite expectation and hence its value is finite almost surely.   
%\[
%\mathbb E\left[\sum_{N=1}^\infty|u_{ij}^{(N)}|^4\right] < \infty. 
%\] 
\end{proof}

\begin{Lem}[{\cite[Chapter 1, Section 4, Problem 1]{Bil68}}]   \label{lem:conv_law}
 Let $X_N, X, Y_N, Y$ be $\C$-valued random variables, $N\in\N$. 
If $X_N  \overset{\text{law}}{\longrightarrow} X$ and $Y_N  \overset{\text{prob}}{\longrightarrow} 0$ then $X_N + Y_N \overset{\text{law}}{\longrightarrow} X$ and $X_N Y_N  \overset{\text{prob}}{\longrightarrow}0$. 
\end{Lem}

Below we denote $\|X\|:= \sqrt{\Tr[X^{*} X]}$ for $X \in M_r(\C)$. 

\begin{Lem}[Eigenvalues of perturbed matrices] \label{lem:distinct}
Let $r \ge 2$ and $\Lambda = \diag(\lambda_1,\lambda_2,\dots, \lambda_r) \in M_r(\C)$, where  $\lambda_1,\lambda_2,\dots, \lambda_r$ are distinct complex numbers. 
Then there exist homogeneous polynomials $\Pi_{p,k}(X)$ $(k\in\N, p\in[r])$ of degree $k$ on the $r^2$ complex variables $X=\{x_{ij}\}_{i,j\in[r]}$ and a constant $C>0$ such that for all $X \in M_r(\C)$ with $\|X\| <C$ the eigenvalues of the perturbed matrix $\Lambda + X$ can be expressed as the absolutely convergent series expansions
\begin{equation} \label{eq:ev_expansions}
\lambda_p  + \sum_{k=1}^\infty \Pi_{p,k}(X), \qquad p \in [r].   
\end{equation}
In particular, the first two terms $\Pi_{p,1}$ and $\Pi_{p,2}$ are given by
 \[\Pi_{p,1}(X) =  x_{pp}, \qquad \Pi_{p,2}(X) = \sum_{i\ne p} \frac{ x_{ip} x_{pi}}{\lambda_p - \lambda_i}. \]
\end{Lem}
%\begin{Exam} As the proof shows, . \]
%\end{Exam}
\begin{proof} The function $f(z, X):= \det (z 1_r - (\Lambda + X))$ is a polynomial of $r^2+1$ variables, $f(\lambda_p, 0) = 0$ and $\partial_z f(\lambda_p, 0) \ne0$; the last condition holds by the assumption of simplicity. By the holomorphic implicit function theorem \cite[p.\ 34]{FG02}, there exist neighborhoods $U_p$ of $\lambda_p \in \C$ and $V_p$ of $0 \in M_r(\C)$ and holomorphic function $\mu^\Lambda_p\colon V_p \to U_p$ such that 
\[
\{(z,X) \in U_p \times V_p \mid f(z,X)=0\} = \{(\mu^\Lambda_p(X),X) \mid X \in V_p\}. 
\]
As being a holomorphic function of several variables, $\mu^\Lambda_p$ has an absolutely convergent series expansion in a neighborhood of $0$ and hence is of the form \eqref{eq:ev_expansions}, as desired.  

The formulas for $\Pi_{p,1}$ and $\Pi_{p,2}$ follow from straightforward calculus, i.e., taking partial derivatives in the identity $f(\mu^\Lambda_p(X), X)=0$ with respect to $x_{ij}$'s and evaluating at $X=0$ yields formulas for $\partial_{x_{ij}}\mu^\Lambda_p(0), \partial^2_{x_{ij}x_{k\ell}}\mu^\Lambda_p(0)$ for $i,j, k,\ell \in [r]$.  
%As a consequence of the argument principle in complex analysis, the solutions to the equation $\det (z I_r - (\Lambda + X))=0$ converge to those to $\det (z I_r - \Lambda)=0$ as $\|X\|\to 0$. Therefore, $\mu_1(X), \mu_2(X), \dots, \mu_r(X)$ can be labeled in such a way that $\mu_p(X)\to \lambda_p$  as $\|X\|\to 0$ for all $p$. 
\end{proof}

%%%%%%%%%%%%%%%%%%%%%%%%%%%%%%%%%%%%%%%%%%%%%%%%%%%
%%%%%%%%%%%%%%%%%%%%%%%%%%%%%%%%%%%%%%%%%%%%%%%%%%%
%%%%%%%%%%%%%%%%%%%%%%%%%%%%%%%%%%%%%%%%%%%%%%%%%%%

\section{Simple eigenvalues: a general algorithm and examples}

%It is clear that the rank of $X_N$ is at most $2r$. %In fact, we can show that the rank is exactly $r+s$ almost surely from the formula \eqref{eq:CP} below. 
%Let 
%$(\lambda^{(N)}_1, \lambda^{(N)}_2, \dots, \lambda_{r+s}^{(N)})$ be the non-zero eigenvalues of $X_N$ put in the non-increasing order. According to %\cite{CHS18}, the almost sure convergence 
%\[
%\lim_{N\to\infty}(\lambda^{(N)}_1, \lambda^{(N)}_2, \dots, \lambda_{r+s}^{(N)}) = (\gamma_1,\gamma_2, \dots, \gamma_{r+s})
%\] 
%holds, where $(\gamma_i)_{i=1}^{r+s}$ is the rearrangement of the multiset $\{\alpha_i, \beta_j: i\in [r], j \in [s]\}$ in the non-increasing manner. Note that we do not need to treat the positive eigenvalues and the negative ones separately because the rank is bounded. 

 %%%%%%%%%%%%%%%%%%%%%%%%%%%%%%%%%%%%%%%%%%%%%%%%%%%
 \subsection{Algorithm and proofs of Theorems \ref{thm:converge.a.s.} and \ref{thm:fluctuations}} \label{sec:algorithm}

The algorithm for computing the fluctuations of eigenvalues of \eqref{eq:model} is what follows. We specialize in the case $k=\ell=1$, which lightens the notation but does not decrease the essence. Let $A:= A_1$ and $B:= B_1$.   
We first choose the basis (also regarded as an $N \times N$ matrix)
\begin{equation} \label{eq:basis}
\mathbf{B}:= (e_1, e_2, \dots, e_r, u_1, u_2, \dots, u_r, e_{r+1}, e_{r+2}, \dots, e_{N-r}),
\end{equation}
 where $u_i$ is the $i$-th column vector of $U^{(N)}$.  
 Note that $\mathbf{B}$ is a basis with probability one since the truncated Haar unitary $\widehat{U}$ has the continuous density in $M_{r}(\C)$ due to Theorem \ref{lem:TruncatedUnitary}  and the set of singular matrices is a null set with respect to Lebesgue measure.  
%{\color{red} Give an account that this is a basis. This is equivalent to showing that $[U]_{I,I}\ne0$ almost surely, where $I=\{N-s+1, \dots, N\}$.}

%Let us investigate $P(X_1 U^*, X_2 U^*, \dots, X_p U^*,  U Y_1, U Y_2, \dots, UY_q)$, where $P(x,y)$ is a noncommutative polynomials without constant term.  We specialize to $r=s$; for example $\hU$ is now the truncated matrix $(u_{i,j})_{i,j \in[r]}$. 

The matrix representations of $A U^*$ and $U B$ with respect to the basis $\mathbf B$ are given by 
\[
 A':=  \begin{pmatrix}
   \hA \hU^* & \hA & * \\
   0 & 0 & 0 \\ 
   0 & 0 & 0 \\  
       \end{pmatrix} 
       \qquad \text{and} \qquad 
B':=       \begin{pmatrix}
  0& 0 & 0 \\
   \hB &  \hB \hU & 0 \\ 
   0 & 0 & 0 \\  
       \end{pmatrix},      
\]
respectively. Let $P(x,y)$ be a polynomial without a constant term in noncommuting indeterminates $x,y$. Because $P(AU^*,U B) = \mathbf B P(A',B')\mathbf B^{-1}$, it suffices to compute the eigenvalues of the matrix $P(A',B')$ which is of the form  
\begin{equation} \label{eq:corner}
\underbrace{\begin{pmatrix}
   P_{11}(\hA, \hB) & P_{12}(\hA, \hB) & * \\
   P_{21}(\hA, \hB) & P_{22}(\hA, \hB) & * \\ 
   0 & 0 & 0 \\  
       \end{pmatrix}}_{=:M} 
       + 
      \underbrace{\begin{pmatrix} 
     O(\|\hU\|) &O(\|\hU\|)  & * \\   
   O(\|\hU\|) & O(\|\hU\|)  & * \\
   0 & 0 & 0 \\  
       \end{pmatrix} }_{=: V},  
\end{equation}
where $ P_{ij}(\hA, \hB)  ~(i,j =1,2)$ does not contain $\hU$. 
It suffices to work on the submatrix $\tilde M + \tilde V$ consisting of the first $2r$ row and columns of $M+V$.  The eigenvalues of $\tilde M + \tilde V$ are called the \emph{nontrivial eigenvalues} of $P(AU^*,U B)$.

\begin{proof}[Proof of Theorem \ref{thm:converge.a.s.}]The entries of the matrix $\tilde V$ are polynomials on entries of $\widehat A, \widehat B, \widehat U,\widehat U^*$ without a constant term with respect to  $\widehat U,\widehat U^*$, so that, by Lemma \ref{lem:as}, they converge to 0 almost surely.  This implies that the eigenvalues of $\tilde M + \tilde V$ converge to those of $\tilde M$, which can be easily proved by applying the argument principle in complex analysis to the characteristic polynomials (this is a simple case of Lemma \ref{lem:conv_poly} below where $d_N$ are all equal to $d=2r$).      
%In order to complete the proof of Theorem \ref{thm:converge.a.s.}, it is convenient to take a Skorohod representation of $\widehat U$ so that $\widehat U$ (and hence $\tilde V$) converges to 0 almost surely; then the characteristic polynomial of $\tilde M + \tilde V$ converges almost surely to that of $\tilde M$, which implies the convergence of eigenvalues from a well known fact in complex analysis (see Lemma \ref{lem:conv_poly} below). 
%This finishes the proof of Theorem \ref{thm:converge.a.s.}. %{\color{red} If $\widehat U$ converges to 0 almost surely, then the proof will be simpler. }
\end{proof}

The eigenvalues of $\tilde M$ and $\tilde M + \tilde V$ are denoted by $\{\mu_i\}_{i=1}^{2r}$ and $\{\mu_i^{(N)}\}_{i=1}^{2r}$ respectively according to the notation of Theorem \ref{thm:fluctuations}. 

\begin{proof}[Proof of Theorem \ref{thm:fluctuations}]
%Calculations of fluctuations become more involved when the main part $\tilde M$ has multiple eigenvalues, so that we assume hereafter that  $\tilde M$ has distinct $2r$ eigenvalues (the case of multiple eigenvalues will be discussed in Section \ref{sec:multiple}). 
Since $\{\mu_i\}_{i=1}^{2r}$ have no multiplicities by the assumption,  there exists an invertible matrix $\tilde R$ of size $2r$ such that $\tilde R^{-1} \tilde M \tilde R= \diag(\mu_1,\mu_2,\dots, \mu_{2r})$. Apparently the eigenvalues of $\tilde M + \tilde V$ are exactly those of the matrix 
\[
\diag(\mu_1,\mu_2,\dots, \mu_{2r}) + \tilde R^{-1} \tilde V \tilde R. 
\]
Then Lemmas \ref{lem:TruncatedUnitary} and \ref{lem:distinct} lead to Theorem \ref{thm:fluctuations} as desired. Indeed, Lemma \ref{lem:distinct} yields
\[
\mu_i^{(N)} = \mu_i  + \sum_{k=1}^\infty \Pi_{i,k}( \tilde R^{-1} \tilde V \tilde R),
\]
which absolutely converges for sufficiently large $N$ since $\tilde V \to 0$ a.s. The RHS is a power series on variables $\{u_{ij}, \overline{u_{ij}} \mid i,j\in[r]\}$ and can be regrouped into 
\[
\mu_i^{(N)} = \mu_i  + \sum_{k=1}^\infty Q_{i,k}(\widehat U, \widehat U^*),  
\]
where $Q_{i,k}(X, X^*)$ is a homogeneous polynomial of degree $k$ on commuting indeterminates $X=\{x_{ij}\}_{i,j\in[r]}$ and $X^*=\{\overline{x_{ji}}\}_{i,j\in[r]}$. 
We set 
\begin{equation}\label{eq:exponent2}
\kappa_i := \inf \{k\in\N \mid Q_{i,k}(X, X^*) \ne0\}.
\end{equation}
 If $\kappa_i=\infty$ then $\mu_i^{(N)} = \mu_i$ a.s.  If $\kappa_i <\infty$ then we can easily prove by Lemma \ref{lem:conv_law} that $N^{\frac{\kappa_i}{2}}(\mu_i^{(N)} -\mu_i)$ converges in law to $Q_{i,\kappa_i}(Z,Z^*)$, where $Z$ is a standard complex Gaussian matrix. Moreover, this convergence holds jointly for all $i$ for which $\kappa_i <\infty$. Note that $Q_{i,\kappa_i}(Z,Z^*)$ is a nonzero random variable; indeed, because the set $S:= \{X \in M_r(\C)\mid Q_{i,\kappa_i}(X,X^*)=0\}$ is a null set with respect to the Lebesgue measure and $Z$ has a probability density function, the probability of the event $Z \in S$ is zero. This implies that for all $\kappa> \kappa_i$, $N^{\frac{\kappa}{2}}(\mu_i^{(N)} -\mu_i)$ does not converge in law and hence the definitions \eqref{eq:exponent1} and \eqref{eq:exponent2} coincide. 
\end{proof}

%Note that  $\kappa_i$ is a random variable but it is a constant almost surely since $Q_{i,k}(\widehat U) =0$ implies the coefficients of $Q_{i,k}(X)$ are all zero (except on a null set).  
%This finishes the proof of Theorem \ref{thm:fluctuations}.
%Using Lemma \ref{lem:distinct}, we can calculate, in principle, the eigenvalues of this matrix up to the first order fluctuations. 

%\begin{proof}[Proof of Theorem \ref{thm:fluctuations}]
 % With the additional assumption of the simplicity of $\tilde M$, Lemmas \ref{lem:TruncatedUnitary} and \ref{lem:distinct} directly lead to the desired result.
%\end{proof}

\subsection{The case  $UA+AU^*$}
%In the examples below, the eigenvalues can be explicitly computed up to the first order fluctuations. 

\begin{Prop}\label{prop:UA+AU*}
The $2r$ nontrivial eigenvalues of the matrix
\begin{align*}
U A + A U^*, 
\end{align*}
where 
\[
 A = 
\begin{pmatrix}  \hA & 0  \\
   0 & 0 \\ 
\end{pmatrix}  \in M_N(\C)
  \quad \text{and} \quad      
 \hA  = \diag(\alpha_1,\alpha_2, \dots, \alpha_r), \quad \alpha_1, \alpha_2,\dots, \alpha_r \in \C,
\]
converge to $\{ \alpha_{i}, -\alpha_i\}_{i=1}^r$ a.s. In addition, if these $2r$ limiting numbers are all distinct (which implies that they are nonzero) then the nontrivial eigenvalues of $U A + A U^*$ are of the forms 
\[
\alpha_i + \frac{\alpha_i}{\sqrt{N}} \Re[\sqrt{N}u_{ii}^{(N)}] + \frac{\omega_{i,+}^{(N)}}{N} \qquad \text{and}\qquad - \alpha_i + \frac{ \alpha_i}{\sqrt{N}} \Re[\sqrt{N}u_{ii}^{(N)}] + \frac{\omega_{i,-}^{(N)}}{N}, \qquad i \in [r], 
\]
where $\{\omega_{i,\pm}^{(N)}\}_{i=1}^r$ are random variables converging in law to $\C$-valued random variables.  See Figure \ref{fig:H5} for a simulation. 
\end{Prop}
\begin{proof}

With respect to the basis $\mathbf B$ introduced in \eqref{eq:basis}, the matrix $U A + A U^*$ has the matrix representation 
\[
T :=  \begin{pmatrix}
   \hA \hU^* & \hA & * \\
     \hA &  \hA \hU & 0 \\ 
   0 & 0 & 0 \\  
       \end{pmatrix} 
     =    \begin{pmatrix}
  0& \hA & 0 \\
     \hA & 0 & 0 \\ 
   0 & 0 & 0 \\  
       \end{pmatrix} 
       +
    \begin{pmatrix}
   \hA \hU^* &  0   & * \\
      0&  \hA \hU & 0 \\ 
   0 & 0 & 0 \\  
       \end{pmatrix}.  
     \]
The main part of $T$ can be diagonalized by the orthogonal matrix     
     \[
R:=      \begin{pmatrix}
   \frac1{\sqrt{2}}1_r & \frac1{\sqrt{2}}1_r & 0 \\
    \frac1{\sqrt{2}}1_r&  - \frac1{\sqrt{2}}1_r & 0 \\ 
   0 & 0 & 1_{N-2r} \\  
       \end{pmatrix}
     \]
in such a way that 
\[
R^{-1} T R = \begin{pmatrix}
   \hA & 0 & 0 \\
     0 & -  \hA & 0 \\ 
   0 & 0 & 0 \\  
       \end{pmatrix} 
        +  
        \begin{pmatrix}
  \hA  \frac{\hU + \hU^*}{2} &  \hA  \frac{ \hU^* - \hU}{2} & * \\
    \hA  \frac{ \hU^* - \hU}{2}  &   \hA  \frac{\hU + \hU^*}{2} & * \\ 
   0 & 0 & 0 \\  
       \end{pmatrix}. 
\]
Suppose further that the $2r$ numbers $\pm \alpha_1, \pm \alpha_2, \dots, \pm\alpha_r$ are all distinct. As a consequence of Lemma \ref{lem:distinct}, the nontrivial eigenvalues of $T$, denoted by $\mu_i^{(N)}, \nu_i^{(N)}, i \in [r],$ are of the forms 
\begin{align*}
\mu_i^{(N)} 
&= \alpha_i + \frac1{2} (\hA  (\hU + \hU^*) )_{i i} + O(\|\hU\|^2) = \alpha_i +  \alpha_i\Re(u_{i i})  + O(\|\hU\|^2), \qquad i \in [r], \\
\nu_i^{(N)} 
&= - \alpha_i  + \alpha_i  \Re(u_{i i})  + O(\|\hU\|^2), \qquad i \in [r]. 
\end{align*}
\end{proof}
\begin{Rem} 
With the help of Lemma \ref{lem:TruncatedUnitary} and Lemma \ref{lem:conv_law}, this proposition implies that,  as $N\to\infty$,  
\[
\sqrt{N}(\mu_i^{(N)}  - \alpha_i)  \overset{\text{law}}{\longrightarrow} \frac{\alpha_i}{\sqrt{2}} x_i  \quad \text{and} \quad \sqrt{N}(\nu_i^{(N)}  + \alpha_i)  \overset{\text{law}}{\longrightarrow}  \frac{\alpha_i}{\sqrt{2}}  x_i \quad \text{for all} \quad i\in [r],     
\]  
where $\{x_i\}_{i=1}^r$ are iid random variables, each distributed as $N(0,1)$. 

\end{Rem}

\begin{figure}[h!]
\begin{center}
\begin{overpic}[width=7cm]{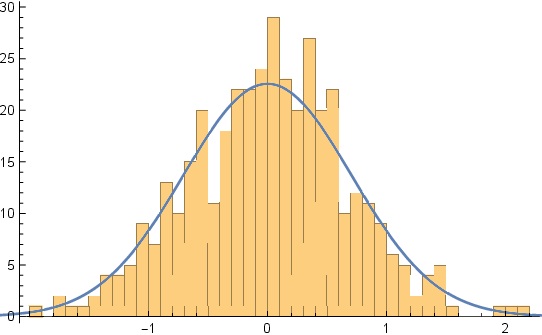}
\end{overpic}
\end{center}
\caption{A histogram for $\sqrt{N}(\mu_1^{(N)} - 4)/4$ (made of 400 samples), where $\mu_1^{(N)}$ is the eigenvalue near 4 of the matrix model $AU + U^*A $ with $A= \diag(4,2,1, 0,0,\dots,0)$ of size  $N=400$. The appended curve is the probability density function of $N(0,1)$ multiplied by $40$.} \label{fig:H5}
\end{figure}

\subsection{The case $P(A,UBU^*)$ and a proof of Theorem \ref{exam:A-UBU*}}

The specialized model $P(A,UBU^*)$ is easier to analyze than \eqref{eq:model} because the main part of the representation matrix is already diagonalized; see \eqref{eq:pol} below.  In this subsection we modify the definition of $\hU$ to the rectangular truncation $(u_{ij})_{i\in [r], j\in[s]}$.

%Next, we prove Theorem \ref{exam:A-UBU*}, i.e., compute fluctuations of eigenvalues of general polynomials $P(A,UBU^*)$. 

\begin{proof}[Proof of Theorem \ref{exam:A-UBU*}]
Straightforward calculations yield that, with respect to the modified basis 
\[
(e_1, e_2, \dots, e_r, u_1, u_2, \dots, u_s, e_{r+1}, e_{r+2}, \dots, e_{N-s}), 
\] 
 $A$ and $UBU^*$ have the matrix representations    \begin{align*}
  \tA:=   \begin{pmatrix}
   \hA & \hA \hU & 0 \\
   0 & 0 & 0 \\ 
   0 & 0 & 0 \\  
       \end{pmatrix} 
       \qquad \text{and}\qquad 
\tB:=     \begin{pmatrix}
     0&0&0 \\
   \hB \hU^* & \hB  & \ast \\ 
   0 & 0 & 0 \\  
       \end{pmatrix}, 
  \end{align*}
 respectively, where \[
\hA := \diag(\alpha_1,\alpha_2,\dots, \alpha_r)   \in M_r(\C), \qquad \hB := \diag(\beta_1,\beta_2,\dots, \beta_s)   \in M_s(\C). 
\]   
From this it is clear that our model $P(A, U B U^*) $ has $N-r-s$ trivial zero eigenvalues. 

The polynomial $P(x,y)$ can be decomposed into 
\begin{align*}
P(x,y) &= P_1(x) + Q_1(y) + \underbrace{\sum_{k,\ell \ge1} a_{k,\ell} x^k y^\ell}_{=:P_2(x,y)} + \underbrace{\sum_{k,\ell \ge1} b_{k,\ell} y^\ell x^k}_{=:Q_2(x,y)}  \\
&\quad +  \underbrace{\sum_{k,\ell,m \ge1} c_{k,\ell,m} x^k y^\ell x^m}_{=:P_3(x,y)}+  \underbrace{\sum_{k,\ell,m \ge1} d_{k,\ell,m} y^\ell x^k y^m}_{=: Q_3(x,y)}  + R(x,y),  
\end{align*}
where $b_{k,\ell}, c_{k,\ell,m}, d_{k,\ell,m}$ are complex coefficients, and $R$ is a linear combination of monomials of lengths larger than three (the length of the elements $x^k$ and $y^\ell$ is counted as one). 
For $k,\ell, m \ge 1$ we have 
\begin{align*}
& \tA^k =  
 \begin{pmatrix} 
\hA^k & \hA^k \hU & 0 \\
   0 & 0 & 0 \\ 
   0 & 0 & 0 \\  
 \end{pmatrix},  \qquad  
 \tB^\ell = 
   \begin{pmatrix} 
        0 & 0 & 0 \\   
     \hB^\ell \hU^*&  \hB^\ell & \ast \\
   0 & 0 & 0 \\  
       \end{pmatrix}, \\
& \tA^k \tB^\ell =  
 \begin{pmatrix} 
\hA^k \hU \hB^\ell \hU^* & \hA^k \hU \hB^\ell & \ast \\
   0 & 0 & 0 \\ 
   0 & 0 & 0 \\  
 \end{pmatrix},  \qquad  
 \tB^\ell \tA^k = 
   \begin{pmatrix} 
        0 & 0 & 0 \\   
     \hB^\ell \hU^* \hA^k &  \hB^\ell \hU^* \hA^k \hU & 0 \\
   0 & 0 & 0 \\  
       \end{pmatrix}, \\
 &   \tA^k \tB^\ell \tA^m =  
 \begin{pmatrix} 
\hA^k \hU \hB^\ell \hU^* \hA^m &\hA^k \hU \hB^\ell \hU^* \hA^m\hU  & 0 \\
   0 & 0 & 0 \\ 
   0 & 0 & 0 \\  
 \end{pmatrix},  \\ 
 &
 \tB^\ell \tA^k \tB^m = 
   \begin{pmatrix} 
        0 & 0 & 0 \\   
     \hB^\ell \hU^* \hA^k \hU \hB^m  \hU^* &   \hB^\ell \hU^* \hA^k \hU \hB^m   & \ast \\
   0 & 0 & 0 \\  
       \end{pmatrix}.   
       \end{align*}
 With the convention that $c_{k,\ell,0} = a_{k,\ell}, d_{k,\ell,0}= b_{k,\ell}$ and $P_1(x) = \sum_{k\ge1} a_{k,0}x^k, Q_1(y) =\sum_{\ell \ge1} b_{0,\ell} y^\ell$  
 we get 
   \begin{align}
  P(\tA,\tB) 
  &=  
  \begin{pmatrix} 
        P_1(\hA) & P_1(\hA)\hU & 0 \\   
    Q_1 (\hB) \hU^*   & Q_1(\hB)  & \ast \\
   0 & 0 & 0 \\  
  \end{pmatrix} 
  + 
   \sum_{k,\ell \ge1}  
   \begin{pmatrix} 
   a_{k,\ell}   \hA^k \hU \hB^\ell \hU^* &a_{k,\ell}  \hA^k \hU \hB^\ell  & \ast \\   
   b_{k,\ell}   \hB^\ell \hU^* \hA^k  &  b_{k,\ell} \hB^\ell \hU^* \hA^k \hU  & \ast \\
   0 & 0 & 0 \\  
  \end{pmatrix}   
  \notag  \\ 
  & \qquad 
  + \sum_{k,\ell,m\ge1}
  \begin{pmatrix} 
     c_{k,\ell,m}\hA^k \hU \hB^\ell \hU^* \hA^m & 0  & 0 \\   
   0  & d_{k,\ell,m} \hB^\ell \hU^* \hA^k \hU \hB^m  & \ast \\
   0 & 0 & 0 \\  
  \end{pmatrix} 
  \notag \\
   &\qquad + 
  \begin{pmatrix} 
     O(\|\hU\|^{3}) &O(\|\hU\|^{2})  & \ast \\   
   O(\|\hU\|^{2}) & O(\|\hU\|^{3})  & \ast \\
   0 & 0 & 0 \\  
       \end{pmatrix}  \notag   \\ 
   &=   \begin{pmatrix} 
        P_1(\hA) & 0& 0 \\   
  0 & Q_1(\hB)  & \ast \\
   0 & 0 & 0 \\  
       \end{pmatrix} 
       + 
   \sum_{k\ge1, \ell \ge0}   
       \begin{pmatrix} 
    0 & a_{k,\ell}\hA^k \hU \hB^\ell  & \ast \\   
     0  & 0  & \ast \\
   0 & 0 & 0 \\  
       \end{pmatrix} 
       +   
   \sum_{k\ge0 ,\ell \ge1}    \begin{pmatrix} 
    0 & 0& \ast \\   
     b_{k,\ell}\hB^\ell \hU^* \hA^k  & 0  & \ast \\
   0 & 0 & 0 \\  
       \end{pmatrix}    
       \notag    \\ 
  & \qquad +  \sum_{k,\ell\ge1, m\ge0}\begin{pmatrix} 
     c_{k,\ell,m}\hA^k \hU \hB^\ell \hU^* \hA^m & 0  & 0 \\   
   0  & d_{k,\ell,m} \hB^\ell \hU^* \hA^k \hU \hB^m  & \ast \\
   0 & 0 & 0 \\  
       \end{pmatrix} 
      \notag \\ 
      &\qquad 
       + 
   \begin{pmatrix} 
     O(\|\hU\|^{3}) &O(\|\hU\|^{2})  & \ast \\   
   O(\|\hU\|^{2}) & O(\|\hU\|^{3})  & \ast \\
   0 & 0 & 0 \\  
       \end{pmatrix} 
       \label{eq:pol}
    \end{align}
Suppose that the $r+s$ eigenvalues of the main part 
\[
 \begin{pmatrix} 
        P_1(\hA) & 0 \\   
    0   & Q_1(\hB)    
       \end{pmatrix} 
       \]
are all distinct. Note that these eigenvalues are $P_1(\alpha_i), i \in[r]$ and $Q_1(\beta_j), j\in[s]$.  According to Lemma \ref{lem:distinct} the eigenvalues $\mu^{(N)}_1,\mu^{(N)}_2,\dots, \mu^{(N)}_{r+s}$ of the first $(r+s)$-dimensional corner of $P(\tA,\tB)$ (called the nontrivial eigenvalues) are of the form 
\begin{align}
\mu^{(N)}_i 
&= P_1(\alpha_i) +  \sum_{k,\ell\ge1, m\ge0}  c_{k,\ell,m}  (\hA^k \hU \hB^\ell \hU^* \hA^m)_{i,i}    \notag \\
& \qquad+   \sum_{j \in[s]}\frac{1}{P_1(\alpha_i) - Q_1(\beta_j)} \sum_{k\ge1,\ell\ge0} a_{k,\ell} (\hA^k  \hU \hB^\ell )_{i,j}  \sum_{k'\ge0,\ell'\ge1} b_{k',\ell'} (\hB^{\ell'}  \hU^* \hA^{k'} )_{j,i}  + O(\|\hU\|^3)  \notag \\
&= P_1(\alpha_i)+   \sum_{k,\ell\ge1, m\ge0}  c_{k,\ell,m} \sum_{j\in[s]} \alpha_i^{k+m}\beta_j^\ell |u_{i,j}|^2    \notag  \\
& \qquad+   \sum_{j \in[s]}\frac{1}{P_1(\alpha_i) - Q_1(\beta_j)} \sum_{k\ge1,\ell\ge0} a_{k,\ell} \sum_{k'\ge0,\ell'\ge1} b_{k',\ell'}  \alpha_i^{k+k'} \beta_j^{\ell+\ell'} |u_{i,j}|^2+ O(\|\hU\|^3)     \notag \\
&= P_1(\alpha_i) +    \sum_{j\in[s]} \left( P_2(\alpha_i,\beta_j) + P_3(\alpha_i,\beta_j) +  \frac{[P_1(\alpha_i) + P_2(\alpha_i,\beta_j)]  [Q_1(\beta_j) + Q_2(\alpha_i,\beta_j)]}{P_1(\alpha_i) - Q_1(\beta_j)} \right)  |u_{i,j}|^2    \notag   \\
&\qquad + O(\|\hU\|^3)   \label{eq:mixture_exponential1}
\end{align}
for $i\in[r]$, and similarly,  
\begin{align}
\mu^{(N)}_{r+j} 
&= Q_1(\beta_j) +    \sum_{i\in[r]} \left( Q_2(\alpha_i,\beta_j) + Q_3(\alpha_i,\beta_j) -  \frac{[P_1(\alpha_i) + P_2(\alpha_i,\beta_j)]  [Q_1(\beta_j) + Q_2(\alpha_i,\beta_j)]}{P_1(\alpha_i) - Q_1(\beta_j)} \right)  |u_{i,j}|^2  \notag  \\ 
&\qquad + O(\|\hU\|^3)  \label{eq:mixture_exponential2}
\end{align}
for $j \in[s]$.  
The random variables $N|u_{ij}^{(N)}|^2$ converge in law to $(\Re(z_{ij}))^2 + (\Im(z_{ij}))^2$ which follows the exponential distribution $e^{-x}\,dx, x>0$.
%Therefore, the random variable $N(\mu_i^{(N)} - P_1(\alpha_i))$ converges in law to a certain mixture of exponential distributions. 
\end{proof}

\begin{figure}[h!]
\begin{center}
\begin{overpic}[width=7cm]{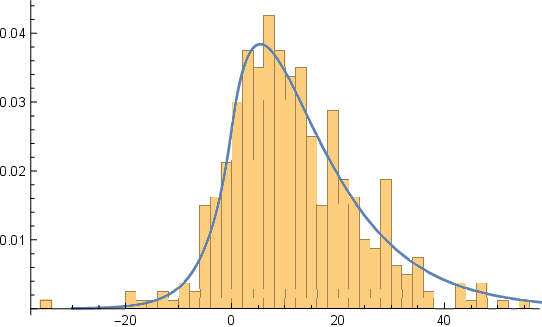}
\end{overpic}
\end{center}
\caption{A histogram for $N(\mu_2^{(N)} - 2)$  (400 samples, normalized to have area 1), where $\mu_2^{(N)}$ is the eigenvalue near $2$ of the model $A + UBU^* +  AUBU^*A + UBU^*AUBU^*$ with $A= \diag(5,2,1, 0,0,\dots,0), B= \diag(4,3,-1,0,0,\dots,0)$ of size $N=400$, together with the theoretical limiting probability density function $\frac{21}{800}e^{3x/14} 1_{(-\infty,0)}(x) + \frac{3}{800}(-25e^{-x/6} + 32 e^{-x/12}) 1_{[0,\infty)}(x)$. } \label{fig:H1}
\end{figure}

\begin{figure}[h!]
\begin{center}
\begin{overpic}[width=7cm]{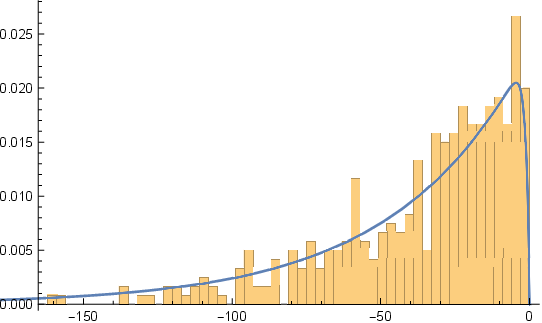}
\end{overpic}
\end{center}
\caption{A histogram for $N(\mu_1^{(N)} - 2)$ (400 samples, normalized to have area 1), where $\mu_1^{(N)}$ is the eigenvalue near 2 of the model $A + UBU^* +  AUBU^*  + UBU^*A   +  \frac{1}{2}(AUBU^*A + UBU^*AUBU^*)$ with $A= \diag(2,1,-1, 0,0,\dots,0), B= \diag(4,-0.2,0,0,\dots,0)$ of size $N=400$, together with the theoretical limiting probability density function $\frac{55}{2352}(e^{x/44} -e^{55x/68}) 1_{(-\infty,0)}(x)$. } \label{fig:H2}
\end{figure}

%%%%%%%%%%%%%%%%%%%%%%%%%%%%%%%%%%%%%%%%%%%%%%%%%%%
%%%%%%%%%%%%%%%%%%%%%%%%%%%%%%%%%%%%%%%%%%%%%%%%%%%
%%%%%%%%%%%%%%%%%%%%%%%%%%%%%%%%%%%%%%%%%%%%%%%%%%%
\section{Eigenvalues with multiplicities: examples} \label{sec:multiple}

When the main term of \eqref{eq:corner} has multiple eigenvalues, a general algorithm for computing fluctuations would be complicated (Lemma \ref{lem:distinct} works only for simple eigenvalues). Abandoning the general case, we work with the specific model
\begin{equation}\label{eq:sum_model}
X = A + U B U^*, 
\end{equation}
where 
 \begin{align*}
A &:= \diag(\alpha_1,\alpha_2,\dots, \alpha_r, 0,0,\dots,0)  \in M_N(\R),  \\
 B&:= \diag(\beta_1,\beta_2,\dots, \beta_s, 0,0,\dots,0)   \in M_N(\R), \\
 & \alpha_i,\beta_j \in \R\setminus\{0\}, i \in [r], j \in [s].  
\end{align*}
It is easy to see (e.g., \ from \eqref{eq:repr1}) that the limiting eigenvalues of $X$ are $\alpha_i, \beta_j, i \in [r], j\in [s]$ and the others are all zero. 

Even for this specific model, a general algorithm for computing fluctuations looks too difficult. We deal with further special cases.

\subsection{Convergence of polynomials and convergence of roots}

Our analysis of the fluctuations of eigenvalues of the sum model \eqref{eq:sum_model} is based on the characteristic polynomials. 
The following fact is essential to deal with eigenvalues with multiplicities and is a simple consequence of the argument principle in complex analysis. 

\begin{Lem}\label{lem:conv_poly}
 Let $P(z), P_N(z), N \in\N$ be polynomials with complex coefficients such that $P\not\equiv 0$. Let $d:= \deg P(z), d_N := \deg P_N(z)$ and assume that $\sup_{N\in\N} d_N<\infty$. We  denote by $\lambda_i, i\in [d]$ the roots of $P(z)$ counting multiplicities. Suppose that $P_N$ converges to $P$ pointwisely on $\C$. Then $d_N \ge d $ for sufficiently large $N$ and there is a suitable  labeling of the roots of $P_N(z)$ counting multiplicities, denoted by $\lambda_i^{(N)}, i \in[d_N]$, such that 
 \begin{enumerate}[label=\rm(\roman*),leftmargin=1cm]
\item\label{item1} $\displaystyle\lim_{N\to\infty} \lambda_i^{(N)} = \lambda_i$ for all  $i \in [d]$, 
\item\label{item2} $\displaystyle\lim_{N\to\infty} \max_{d+1 \le i \le d_N}|\lambda_i^{(N)}| =\infty$. 
\end{enumerate} 
When $P$ is a nonzero constant, we understand that $d=0$ and assertion \ref{item1} must be deleted. On the other hand, when $d_N = d$ then the number $\max_{d+1 \le i \le d_N}|\lambda_i^{(N)}| $ is to be interpreted as $\infty.$ 
\end{Lem}

%\begin{Rem} The converse is not true; take $P_N(z) = (z-N)z$ and $P(z)=z$ for example. 
%\end{Rem}

%The following example will help see the situation when $\deg P_N(z)$ is strictly larger than $\deg P(z)$. 

\begin{Exam} Let 
\[
P_N(x) = \frac{(-1)^N}{N} x^2 + \left(1-\frac{(-1)^N}{N^2}\right)x - \frac1{N} = \frac1{N}((-1)^N x + N)\left(x-\frac1{N}\right).
\] Then $P_N(x)$ converges to $P(x)=x$ pointwise. The root $x=(-1)^{N-1}N$ tends to $\pm\infty$, while the root $x=\frac1{N}$ converges to $0$ which is the root of $P(x)$. 
\end{Exam}

%%%%%%%%%%%%%%%%%%%%%%%%%%%%%%%%%%%%%%%%%%%%%%%%%%%
\subsection{The characteristic polynomial}

\begin{Notation*}
For an $m \times n$ matrix $C=(c_{ij})_{i \in [m], j \in [n]}$ and two subsets $I \subseteq [m], J \subseteq [n]$ of the same cardinality we let $[C]_{I,J}$ be the determinant of the submatrix $(c_{ij})_{i \in I, j \in J}$. As convention, we also set $[C]_{\emptyset, \emptyset}:=1$.  
\end{Notation*}

According to Example \ref{exam:A-UBU*}, with respect to the basis 
\[
(e_1, e_2, \dots, e_r, u_1, u_2, \dots, u_s, e_{r+1}, e_{r+2}, \dots, e_{N-s})
\]
the matrix $X$ in \eqref{eq:sum_model} has the representation matrix  
    \begin{align} \label{eq:repr1}
   \begin{pmatrix}
   \hA & \hA \hU & 0 \\
  \hB \hU^* & \hB  & \ast \\ 
   0 & 0 & 0 \\  
       \end{pmatrix},  
   \end{align}
where  $\hA = \diag(\alpha_1 ,\alpha_2, \dots, \alpha_r)$,  $\hB= \diag (\beta_1,\beta_2,\dots, \beta_s)$ and $\hU = (u_{ij})_{i \in [r], j\in [s]}$. Let $  \widetilde X$ be the submatrix of \eqref{eq:repr1} consisting of the first $(r+s)$ rows and columns. 
The characteristic polynomial of  $\widetilde X$ is given by 
\begin{align*}
%&\lambda^{N-\Rk-s}\sum_{n=0}^{\min\{r,s\}} \sum_{\substack{I \subseteq [r], J \subseteq [s] \\ \# I = \# J = n}} \prod_{i \in [r]\setminus I} (\lambda-\alpha_i) \prod_{j \in [s]\setminus J}(\lambda-\beta_j) \det\left[(-\alpha_i u_{ij})_{\substack{i\in I \\ j\in J}}\right]  \det\left[(-\beta_j\overline{u_{ij}})_{\substack{i\in I \\ j\in J}}\right] \notag \\
\varphi_N(\lambda):=\sum_{n=0}^{\min\{r,s\}}(-1)^n \sum_{\substack{I \subseteq [r], J \subseteq [s] \\ \# I = \# J = n}} \left\{\prod_{i \in [r]\setminus I} (\lambda-\alpha_i) \prod_{j \in [s]\setminus J}(\lambda-\beta_j)  \prod_{i\in I} \alpha_i \prod_{j\in J} \beta_j \right\} \left|[\widehat{U}]_{I,J}\right|^2,  
\end{align*}
which is a direct consequence of the definition of determinant by prescribing the fixed points of permutations.  Investigating this polynomial will reveal fluctuations of the eigenvalues. 
%We often disregard the trivial factor $\lambda^{N-\Rk-s}$ and work only on 
%\begin{equation}\label{eq:CP}
%\varphi_N(\lambda):=\sum_{n=0}^{\min\{r,s\}}(-1)^n \sum_{\substack{I \subseteq [r], J \subseteq [s] \\ \# I = \# J = n}} \left\{\prod_{i \in [r]\setminus I} (\lambda-\alpha_i) \prod_{j \in [s]\setminus J}(\lambda-\beta_j)  \prod_{i\in I} \alpha_i \prod_{j\in J} \beta_j \right\} \left|[\widehat{U}]_{I,J}\right|^2.  
%\end{equation}

%%%%%%%%%%%%%%%%%%%%%%%%%%%%%%%%%%%%%%%%%%%%%%%%%%%
\subsection{Multiplicities within $A$ (and/or within $B$)} \label{sec:multiple1}
Suppose first that $\alpha_1 = \alpha_2 = \cdots = \alpha_\mm$ for some $\mm \in [r]$ and that none of $\alpha_{\mm+1}, \alpha_{\mm+2}, \dots, \alpha_r$, $ \beta_1, \dots, \beta_s$ equals $\alpha_1$.

%This assumption does not change any conclusion about convergence in law.  
%The details are omitted if arguments are similar to those in the previous subsection. 

\subcases{Convergence of rescaled characteristic polynomial.}
By taking the Skorohod representation, we assume for the moment that $\widehat U^{(N)}$ converges to $Z$ almost surely, see Remark \ref{rem:skorohod}. 
(This replacement will be justified later.)
We begin by observing 
\begin{align*}
&N^\mm \varphi_N\left(\alpha_1 + \frac1{N}\tau\right)   \\
& = \tau^\mm \prod_{i=\mm+1}^r(\alpha_1-\alpha_i) \prod_{j=1}^s(\alpha_1 - \beta_j) \\
&\quad + \sum_{n=1}^{\min\{\mm,s\}}(-1)^n \sum_{\substack{I \subseteq [\mm], J\subseteq [s] \\ \#I = \#J = n}} \tau^{\mm-n} \left[\prod_{i=\mm+1}^r(\alpha_1-\alpha_i) \prod_{j\in [s]\setminus J} (\alpha_1-\beta_j) \prod_{i \in I}\alpha_i \prod_{j \in J}\beta_j \right] N^n \left| [\widehat{U}]_{I,J} \right|^2 \\
&\quad + O\left(\frac1{N}\right),
\end{align*}
where $O\left(\frac1{N}\right)$ is a polynomial on $\tau$ of degree not larger than $r+s$ with coefficients of order $O\left(\frac1{N}\right)$ in the usual sense almost surely. 
Since  $N^m\varphi_N\left(\alpha_1 + \frac1{N}\tau\right)$ converges almost surely to  the polynomial
\[
\psi(\tau):=\tau^\mm  + \sum_{n=1}^{\min\{\mm,s\}}(-1)^n \sum_{\substack{I \subseteq [\mm], J\subseteq [s] \\ \#I = \#J = n}} \tau^{\mm-n} \left[\prod_{j \in J} \frac{\alpha_1\beta_j}{\alpha_1-\beta_j} \right] \left| [Z]_{I,J} \right|^2,  
\]
by Lemma \ref{lem:conv_poly}, the polynomial $\tau\mapsto \varphi_N\left(\alpha_1 + \frac1{N}\tau\right)$ has $\mm$ consecutive roots $\delta_{11}^{(N)}\ge  \delta_{12}^{(N)}\ge  \cdots\ge \delta_{1\mm}^{(N)}$ that converge almost surely to the $\mm$ real roots of $\psi$. 

%\subcases{Roots of $\psi$.}
The roots of $\psi$ can be well described as the eigenvalues of a certain random matrix. 
Let $\gamma_{1j}:=\frac{\alpha_1\beta_j}{\alpha_1-\beta_j} $, $\Gamma_1 := \diag(\gamma_{11},\gamma_{12}, \dots, \gamma_{1s})$ and $Z_1=(z_{ij})_{i \in [\mm], j \in [s]}$ be the truncation of $Z$. Then we have 
\begin{align*}
\psi(\tau)
&=\tau^\mm  + \sum_{n=1}^{\min\{\mm,s\}}(-1)^n \tau^{\mm-n} \sum_{\substack{I \subseteq [\mm], J\subseteq [s] \\ \#I = \#J = n}}  [Z \Gamma_1]_{I,J} [Z^*]_{J,I} \\
&= \tau^\mm  + \sum_{n=1}^{\mm} \tau^{\mm-n} (-1)^n\sum_{\substack{I \subseteq [\mm]\\ \#I = n}}  [Z \Gamma_1Z^*]_{I,I} \\
&= \tau^\mm  +\sum_{n=1}^{\mm} \tau^{\mm-n} (-1)^n\sum_{\substack{I \subseteq [\mm]\\ \#I = n}}  [Z_1\Gamma_1Z_1^*]_{I,I} \\
&= \det(\tau I_\mm - Z_1 \Gamma_1 Z_1^*).   
\end{align*}
Note that $[Z \Gamma_1Z^*]_{I,I}=0$ if $s<n =\#I$ because the rank of $Z \Gamma_1Z^*$ is not greater than $s$ and hence we were allowed to replace $\min\{\mm,s\}$ with $\mm$. 

From the discussions above the random vector $(\delta_{11}^{(N)}, \delta_{12}^{(N)}, \dots, \delta_{1\mm}^{(N)})$ converges almost surely to the sequence of the eigenvalues of $Z_1 \Gamma_1 Z_1^*$ (labeled in the decreasing order), and hence, converges in law. %, where all the sequences are put in a non-increasing manner. 

\subcases{Conclusion.}
The convergence in law of the random vector $(\delta_{11}^{(N)}, \delta_{12}^{(N)}, \dots, \delta_{1\mm}^{(N)})$ also holds for the original random matrix model (without taking the Skorohod representation) because the roots of polynomials can be expressed as measurable (in fact, continuous) functions of the coefficients as a consequence of the argument principle so that each $\delta_{1j}^{(N)}$ is a measurable function of $\widehat U^{(N)}$. Convergence in law is a notion completely determined by the law and hence is unchanged by replacing the random variables with other ones with identical laws. 

\begin{Exam} \label{exa:multipleA}
The preceding arguments allow us to calculate the joint distribution of fluctuations when the entries of $A$ and $B$ are of the form 
\begin{align*}
(\alpha_1,\alpha_2,\dots, \alpha_r)&= (\underbrace{\alpha_{1}', \dots, \alpha_{1}'}_{\text{$m_1$ times}}, \underbrace{\alpha_{2}',\dots, \alpha_{2}'}_{\text{$m_2$ times}}, \dots, \underbrace{\alpha_{p}',\dots, \alpha_{p}'}_{\text{$m_p$ times}}), \\
(\beta_1,\beta_2,\dots, \beta_s)&= (\underbrace{\beta_{1}', \dots, \beta_{1}'}_{\text{$n_1$ times}}, \underbrace{\beta_{2}',\dots, \beta_{2}'}_{\text{$n_2$ times}}, \dots, \underbrace{\beta_{q}',\dots, \beta_{q}'}_{\text{$n_q$ times}}), 
\end{align*} where $\alpha_1' \dots, \alpha_p', \beta_1', \dots, \beta_q'$ are distinct. Without loss of generality, we assume that $\alpha_1' > \alpha_2' > \cdots > \alpha_p'$ and similarly $\beta_1' > \beta_2' > \cdots > \beta_q'$. 

Let $\Gamma_k = \diag(\gamma_{k1}, \dots, \gamma_{k s})$ and $H_\ell=\diag(\eta_{1\ell}, \eta_{2\ell,} \dots, \eta_{r\ell})$, where 
\[
\gamma_{kj} = \frac{\alpha_k'\beta_j}{\alpha_k'-\beta_j} \qquad \text{and} \qquad \eta_{i\ell}=\frac{\alpha_i\beta_\ell'}{\alpha_i-\beta_\ell'}
\]
 and let $Z_k, Y_\ell$ be the $m_k \times s$ and $r \times n_\ell$ submatrices of $Z$, respectively, defined by 
\[
Z= \begin{pmatrix} Z_1 \\ Z_2 \\ \vdots \\ Z_p
\end{pmatrix}
= (Y_1,Y_2,\dots, Y_q). 
\]
  Let $\{\rho_{k1}, \rho_{k2}, \dots, \rho_{k, m_k}\}$ and $\{\sigma_{1\ell}, \sigma_{2\ell}, \dots, \sigma_{m_\ell, \ell}\}$ be the eigenvalues of $Z_k \Gamma_k Z_k^*$ and of $Y_\ell^* H_\ell Y_\ell$, respectively, labeled in the decreasing order.   Then the eigenvalues of  $\widetilde X$ are of the form $\alpha_k' + \frac1{N}\delta_{k j}^{(N)}~(j\in[m_k], k \in [p])$ and $\beta_\ell' + \frac1{N}\epsilon_{i\ell}^{(N)} (i \in [n_\ell], \ell\in [q])$ with $\delta_{k 1}^{(N)}\ge \delta_{k 2}^{(N)} \ge \cdots \ge \delta_{k, m_k}^{(N)}$ and  $\epsilon_{1\ell}^{(N)} \ge \epsilon_{2\ell}^{(N)} \ge \cdots \ge \epsilon_{n_\ell,\ell}^{(N)}$ such that  
  \[
\left(  (\delta_{kj}^{(N)})_{j \in [m_k],k \in [p]}, (\epsilon_{i\ell}^{(N)})_{i \in [n_\ell], \ell \in [q]} \right)  \overset{\text{law}}{\longrightarrow}   \left( (\rho_{kj})_{ j \in [m_k],k \in [p]},   (\sigma_{i\ell})_{i \in [n_\ell], \ell \in [q]} \right) 
  \] as random vectors on $\R^{m_1}_\ge \times \R^{m_2}_\ge \times\cdots \times \R^{m_p}_\ge  \times \R^{n_1}_\ge \times \R^{n_2}_\ge \times \cdots \times \R^{n_q}_\ge $. See Figure \ref{fig:H12} for simulations of fluctuations. 
  
  %Note that this result  in the case $m_k=n_\ell=1$ specializes to the result in Subsection \ref{subsec:distinct}. 
 \end{Exam}

 \begin{figure}[h!]
\begin{center}
\begin{minipage}{0.3\hsize}
\begin{center}
\begin{overpic}[width=5cm]{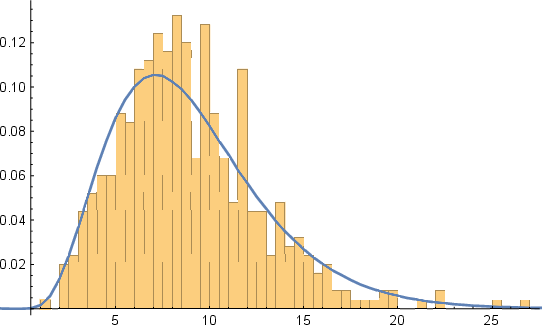}
\end{overpic}
\end{center}
\end{minipage}
\hspace{0mm}
\begin{minipage}{0.3\hsize}
\begin{center}
\begin{overpic}[width=5cm]{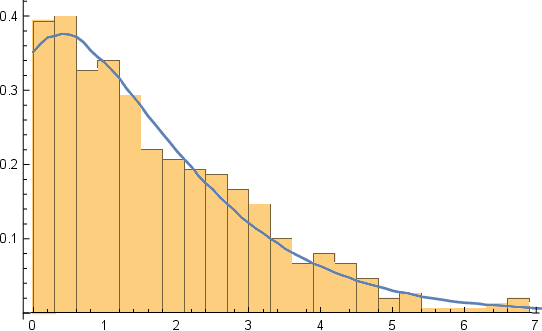}
\end{overpic}
\end{center}
\end{minipage}
\hspace{0mm}
\begin{minipage}{0.3\hsize}
\begin{center}
\begin{overpic}[width=5cm]{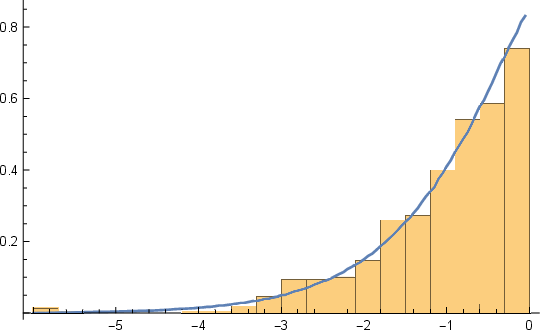}
\end{overpic}
\end{center}
\end{minipage}
\caption{Histograms for $\delta_{1,1}^{(N)}$ (left), $\delta_{1,2}^{(N)}$ (middle) and  $\delta_{1,3}^{(N)}$ (right)  in Example \ref{exa:multipleA} for the model $A+ UBU^*$ with $A= \diag(2,2, 2,0,\dots,0), B= \diag(1,1,-1,0,\dots,0)$ of size $N=400$. The histograms are constructed from 500 samples, and the heights are normalized to have area 1. The appended curves are the probability density functions of $\rho_{1,1}$ (left), $\rho_{1,2}$ (middle), $\rho_{1,3}$ (right), drawn by taking $2\cdot 10^6$ samples and connecting (by line segments) the heights of the histogram.   
} \label{fig:H12}
\end{center}
\end{figure}

%%%%%%%%%%%%%%%%%%%%%%%%%%%%%%%%%%%%%%%%%%%%%%%%%%%%%%%%%%%%%%%%%%%%  
% We could also consider the joint distribution of $(\delta_{ij})$ and fluctuations of $(\beta_j)$, but then the description would be more complicated because then some of $\beta_j$ would be identical. 

%%%%%%%%%%%%%%%%%%%%%%%%%%%%%%%%%%%%%%%%%%%%
\subsection{Common eigenvalues shared by $A$ and $B$}
We assume that  $r=s=2$ and discuss the case where some $\alpha_i$ coincides with some $\beta_j$. The fluctuations are more exotic. 

\begin{Exam} \label{exa:multipleB}
 Suppose that $\alpha_1=\alpha_2= \beta_1 \ne \beta_2$. Without loss of generality, we assume that $\alpha_1 < \beta_2$.  
The characteristic polynomial of $\widetilde X$ is explicitly given by 
\begin{align*}
\varphi_N(\lambda)
&= (\lambda -\alpha_1)^3 (\lambda-\beta_2) - \alpha_1^2 (|u_{11}|^2 +|u_{21}|^2)(\lambda-\alpha_1)(\lambda-\beta_2) \\
&\qquad - \alpha_1\beta_2 (|u_{12}|^2 +|u_{22}|^2)(\lambda-\alpha_1)^2 +\alpha_1^3 \beta_2 |u_{11}u_{22}- u_{12}u_{21}|^2. 
\end{align*}
  Analogously to Section \ref{sec:multiple1}, we  assume for the moment that $\sqrt{N}\widehat{U}^{(N)}$ converges to $Z$ almost surely.

\subcases{Fluctuations of $\alpha_1$.}  Unexpectedly, there are two different scalings. 
Observe first that 
\begin{align}
N^{\frac{3}{2}}\varphi_N\left(\alpha_1 + \frac1{\sqrt{N}}\tau\right)
&=  \tau^3(\alpha_1-\beta_2) - \alpha_1^2 N(|u_{11}|^2 +|u_{21}|^2)\tau(\alpha_1-\beta_2)  + O\left(\frac1{\sqrt{N}}\right), \label{eq:limit2}
\end{align}
which reveals that $\varphi_N(\alpha_1 + \frac1{\sqrt{N}}\tau)$ has two roots $\tau=\delta_1^{(N)}, \delta_2^{(N)}$ that respectively converge almost surely to 
\[
\xi_1:=|\alpha_1|\sqrt{|z_{11}|^2+|z_{21}|^2} \qquad \text{and} \qquad \xi_2:=- |\alpha_1|\sqrt{|z_{11}|^2+|z_{21}|^2}. 
\]
In addition to $\xi_1, \xi_2$, the limiting polynomial of \eqref{eq:limit2} has the root $0$ with multiplicity one. This means that the scaling $1/\sqrt{N}$ is irrelevant for the third root near $\alpha_1$. The right scaling is $N^{-1}$ as we see from 
\begin{align*}
N^2\varphi_N\left(\alpha_1 + \frac1{N}\tau\right)
&=  - \alpha_1^2 N(|u_{11}|^2 +|u_{21}|^2)\tau(\alpha_1-\beta_2) +\alpha_1^3 \beta_2 N^2|u_{11}u_{22}- u_{12}u_{21}|^2 + O\left(\frac1{N}\right) \\
& \to  - \alpha_1^2 (|z_{11}|^2 +|z_{21}|^2)\tau(\alpha_1-\beta_2) +\alpha_1^3 \beta_2 |z_{11}z_{22}- z_{12} z_{21}|^2.  
\end{align*}
Again by Lemma \ref{lem:conv_poly},  the polynomial $\tau\mapsto\varphi_N(\alpha_1 + \frac1{N}\tau)$ has a root  $\tau=\delta_3^{(N)}$ that converges almost surely to the random variable
\[
\xi_3:=\frac{\alpha_1\beta_2}{\alpha_1-\beta_2}\cdot \frac{|z_{11}z_{22}-z_{12}z_{21}|^2}{|z_{11}|^2+|z_{21}|^2}. 
\]
\subcases{Fluctuations of $\beta_2$.}
A similar analysis yields that $\varphi_N$ has a root of the form $\lambda =\beta_2 + \frac1{N}\epsilon^{(N)}$ such that $\epsilon^{(N)}$ converges almost surely to 
\[
\zeta:=\frac{\alpha_1\beta_2}{\beta_2-\alpha_1}(|z_{12}|^2+|z_{22}|^2).  
\]
\subcases{Conclusion.}  Considering the scaling and the signs of limiting random variables, for sufficiently large $N$ (depending on samples),  we have $\beta_2 + \frac1{N}\epsilon^{(N)} >  \alpha_1 + \frac1{\sqrt{N}} \delta_1^{(N)} >  \alpha_1 + \frac1{N} \delta_3^{(N)}  > \alpha_1 + \frac1{\sqrt{N}} \delta_2^{(N)}. $
  Let $\lambda_1^{(N)}\ge\lambda_2^{(N)}\ge\lambda_3^{(N)}\ge\lambda_4^{(N)}$ be the four eigenvalues of $\widetilde X$. We conclude that the random vector
\[
\left(  N (\lambda_1^{(N)}  -\beta_2),    \sqrt{N} (\lambda_2^{(N)}  -\alpha_1),  N (\lambda_3^{(N)}  -\alpha_1),    \sqrt{N} (\lambda_4^{(N)}  -\alpha_1)  \right)  
\] 
converges in law to 
\[
\left(     \frac{\alpha_1\beta_2}{\beta_2-\alpha_1}(|z_{12}|^2+|z_{22}|^2) ,  |\alpha_1|\sqrt{|z_{11}|^2+|z_{21}|^2},    \frac{-\alpha_1\beta_2  |z_{11}z_{22}-z_{12}z_{21}|^2}{(\beta_2-\alpha_1)(|z_{11}|^2+|z_{21}|^2)},  -  |\alpha_1|\sqrt{|z_{11}|^2+|z_{21}|^2}  \right).
\] 
This also holds without taking the Skorohod representation from the corresponding reasoning in Section \ref{sec:multiple1}.  See Figure \ref{fig:H8} for simulations. 
\begin{figure}[h!]
\begin{center}
\begin{minipage}{0.3\hsize}
\begin{center}
\begin{overpic}[width=5cm]{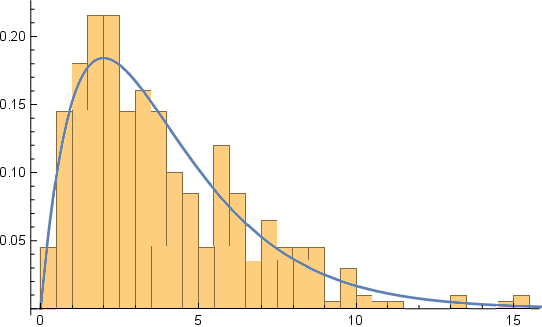}
\end{overpic}
\end{center}
\end{minipage}
\hspace{0mm}
\begin{minipage}{0.3\hsize}
\begin{center}
\begin{overpic}[width=5cm]{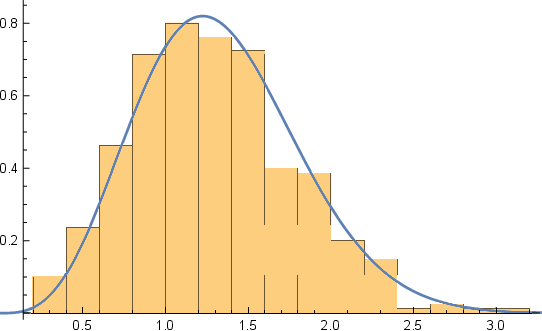}
\end{overpic}
\end{center}
\end{minipage}
\hspace{0mm}
\begin{minipage}{0.3\hsize}
\begin{center}
\begin{overpic}[width=5cm]{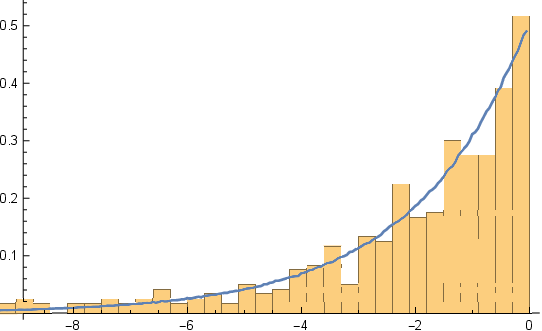}
\end{overpic}
\end{center}
\end{minipage}
\caption{Histograms for $N (\lambda_1^{(N)}  -2)$ (left),  $\sqrt{N} (\lambda_2^{(N)}  -1)$ (middle),   $ N (\lambda_3^{(N)}  - 1)$ (right) in Example \ref{exa:multipleB}  for the matrices $A= \diag(1,1, 0,0,\dots,0), B= \diag(1,2,0,0,\dots,0)$ of size $N=400$.  The histograms are constructed from 400 samples, and the heights are normalized to have area 1. The appended  curves are the probability density functions of $2(|z_{12}|^2+|z_{22}|^2)$ (left,  $(1/4) x e^{-x/2}, x>0$),  $\sqrt{|z_{11}|^2+|z_{21}|^2}$ (middle, $2 x^3 e^{-x^2}, x>0$), $\frac{-2 |z_{11}z_{22}-z_{12}z_{21}|^2}{|z_{11}|^2+|z_{21}|^2}$ (right, drawn by taking $2\cdot 10^6$ samples and connecting the heights of the histogram).} 
 \label{fig:H8}
\end{center}
\end{figure}

\end{Exam}

\begin{Exam} \label{exa:multipleC}
 Suppose that $\alpha_1= \beta_1 < \alpha_2= \beta_2$.  Let $\lambda_1^{(N)}\ge\lambda_2^{(N)}\ge\lambda_3^{(N)}\ge\lambda_4^{(N)}$ be the four eigenvalues of $\widetilde X$. 
A similar technique reveals that the random vector
\[
\left(   \sqrt{N} (\lambda_1^{(N)}  -\alpha_2),      \sqrt{N} (\lambda_2^{(N)}  -\alpha_2),   \sqrt{N} (\lambda_3^{(N)}  -\alpha_1),  \sqrt{N} (\lambda_4^{(N)}  -\alpha_1)  \right)  
\] 
converges in law to 
\[
(|\alpha_2 z_{22}|, -|\alpha_2 z_{22}|,|\alpha_1 z_{11}|, - |\alpha_1 z_{11}|). 
\]
See Figure \ref{fig:H14} for a simulation. 
\begin{figure}[h!]
\begin{center}
\begin{overpic}[width=7cm]{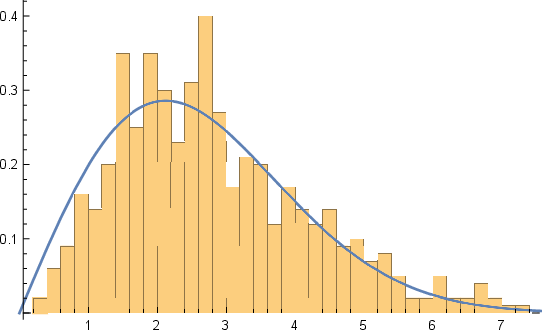}
\end{overpic}
\end{center}
\caption{A histogram for $ \sqrt{N} (\lambda_1^{(N)}  - 3)$ in  Example \ref{exa:multipleC} for the matrices $A= B= \diag (2,3,0,0,\dots,0)$ of size $N=400$,   together with the probability density function of $3 |z_{22}|$ ($(2/9) x e^{-x^2/9}, x>0$). The histogram is constructed from 500 samples, and the heights are normalized to have area 1.} \label{fig:H14}
\end{figure}

\end{Exam}

%%%%%%%%%%%%%%%%%%%%%%%%%%%%%%%%%%%%%%%%%%%%%%%%%%%%%%%%%%%%%%%%%%%%%%%%%%%%%%%
%%%%%%%%%%%%%%%%%%%%%%%%%%%%%%%%%%%%%%%%%%%%%%%%%%%%%%%%%%%%%%%%%%%%%%%%%%%%%%%
\section*{Acknowledgements}
The inserted figures are drawn on Mathematica Version 12.1.1, Wolfram Research, Inc., Champaign, IL. This work is supported by JSPS Grant-in-Aid for Young Scientists 19K14546, Transformative Research Areas (B) grant no.~23H03800, Scientific Research (B) no.~21H00987, Challenging Research (Exploratory)no.~23K17299, Scientific Research (C) no.~23K03133, and JSPS Open Partnership Joint Research Projects grant no.~JPJSBP120209921.
KF was supported by Hokkaido University Ambitious Doctoral Fellowship (Information Science and AI).
FL was supported by the SFB-TRR 195 `Symbolic Tools in Mathematics and their Application' of the German Research Foundation (DFG).

%%%%%%%%%%%%%%%%%%%%%%%%%%%%%%%%%%%%%%%%%%%%%%%%%%%%%%%%%%%%%%%%%%%%%%%%%%%%%%%
%%%%%%%%%%%%%%%%%%%%%%%%%%%%%%%%%%%%%%%%%%%%%%%%%%%%%%%%%%%%%%%%%%%%%%%%%%%%%%%
%Here's the reordered LaTeX bibliography in alphabetical order:

\vspace{4mm}

\begin{flushleft}

Beno\^{i}t Collins 

Kyoto University

Department of Mathematics, Graduate School of Science 

Kitashirakawa Oiwake-cho, Sakyo-ku, Kyoto 606-8502  Japan

Email address: collins@math.kyoto-u.ac.jp

\vspace{4mm}

Katsunori Fujie

Hokkaido University

Department of Mathematics 

North 10 West 8, Kita-ku, Sapporo 060-0810, Japan 

Email address: kfujie@eis.hokudai.ac.jp

\vspace{4mm}

Takahiro Hasebe

Hokkaido University

Department of Mathematics 

North 10 West 8, Kita-ku, Sapporo 060-0810, Japan 

Email address: thasebe@math.sci.hokudai.ac.jp

\vspace{4mm}

Felix Leid 

Saarland University

Department of Mathematics

Postfach 15 11 50,  66041 Saarbr\"{u}cken,  Germany  

Email address: leid@math.uni-sb.de

\vspace{4mm}

Noriyoshi Sakuma 

Nagoya City University

Graduate School of Science 

1, Yamanohata, Mizuho-cho, Mizuho-ku, Nagoya, 467-8501, Japan 

Email address: sakuma@nsc.nagoya-cu.ac.jp

\end{flushleft}

%}% end of \large


\begin{thebibliography}{9}
\bibitem{ArizmendiCebron}
O. Arizmendi, G. C\'ebron, N. Gilliers.
Combinatorics of cyclic-conditional freeness
arXiv:2311.13178
\bibitem{Arizmendi2021} O. Arizmendi, A. Celestino.
Polynomial with cyclic monotone elements with applications to random matrices with discrete spectrum. 
Random Matrices Theory Appl. 10 (2021), no. 2, article no. 2150020 (19 pages). 

\bibitem{Arizmendi2023} O. Arizmendi, T. Hasebe, F. Lehner.
Cyclic independence: Boolean and monotone.
Algebr. Comb. 6 (2023), no. 6, 1697--1734.

\bibitem{BBP2005} 
J. Baik, G. Ben Arous, S. P\'ech\'e.
Phase transition of the largest eigenvalue for nonnull complex sample covariance matrices, Ann. Probab. 33 (5) (2005) 1643--1697. 

\bibitem{Benaych2011} F. Benaych-Georges, R. R. Nadakuditi.
The eigenvalues and eigenvectors of finite, low rank perturbations of large random matrices.
Adv. Math. 227 (2011), no. 1, 494--521.

\bibitem{Bil68} P.\ Billingsley. 
\textit{Convergence of Probability Measures (First Edition)}, John Wiley \& Sons, Inc., New York, 1968.

\bibitem{Bil99} P.\ Billingsley. 
\textit{Convergence of Probability Measures (Second Edition)}, John Wiley \& Sons, Inc., New York, 1999.

\bibitem{C03} B.\ Collins. 
\textit{Int\'egrales matricielles et Probabilit\'es Non-Commutatives}. Math\'ematiques [math]. Universit\'e Pierre et Marie Curie - Paris VI, 2003. Fran\c{c}ais. NNT: tel-00004306.

\bibitem{CapFer2009} M. Capitaine, C. Donati-Martin, D. F\'eral. 
The largest eigenvalues of finite rank deformation of large Wigner matrices: convergence and nonuniversality of the fluctuations, Ann. Probab. 37 (1) (2009) 1--47. 

\bibitem{Cebron2023} G. C\'ebron, A. Dahlqvist, F. Gabriel.
Freeness of type B and conditional freeness for random matrices.
To appear in  Indiana Mathematics Journal, arXiv:2205.01926.

\bibitem{CebronGilliers2024} G. C\'ebron, N. Gilliers.
Asymptotic cyclic-conditional freeness of random matrices.
Random Matrices Theory Appl. 13 (2024), no. 1, Paper No. 2350014, 48 pp.

\bibitem{CHS18} B.\ Collins, T.\ Hasebe, N.\ Sakuma, Free probability for purely discrete eigenvalues of random matrices, J.\ Math.\ Soc.\ Japan 70, No.\ 3 (2018), 1111--1150. 

\bibitem{CLS} B.\ Collins, F.\ Leid, N.\ Sakuma. 
Matrix models for cyclic monotone and monotone independences, arXiv preprint (arXiv:2202.11666). 

\bibitem{CS06} B.\ Collins and P.\ \'Sniady. 
Integration with respect to the Haar measure on unitary,
orthogonal and symplectic group, Comm.\ Math.\ Phys.\ 264 (2006), 773--795. 

\bibitem{FerPe2007} D. F\'eral, S. P\'ech\'e. 
The largest eigenvalue of rank one deformation of large Wigner matrices, Comm. Math. Phys. 272(1) (2007) 185--228.

\bibitem{FG02} K.~Fritzsche and H.~Grauert. 
\textit{From Holomorphic Functions to Complex Manifolds}, Springer-Verlag, New York/ Berlin / Heidelberg, 2002. 
\bibitem{FujieHasebe}
K. Fujie, T. Hasebe.
Free probability of type B prime.
arXiv:2310.14582

\bibitem{Johnstone2001} I. M. Johnstone.
On the distribution of the largest eigenvalue in principal components analysis.
Ann. Statist. 29 (2001), no. 2, 295--327.

\bibitem{MingoSpeicher}
J. A. Mingo and R. Speicher.
\textit{Free probability and random matrices.} Fields Institute Monographs, 35. Springer, New York; Fields Institute for Research in Mathematical Sciences, Toronto, ON, 2017. xiv+336 pp.
\bibitem{Paul2007} D. Paul. 
Asymptotics of sample eigenstructure for a large dimensional spiked covariance model, Statist. Sinica 17 (4) (2007) 1617--1642.

\bibitem{Peche2006} S. P\'ech\'e. 
The largest eigenvalue of small rank perturbations of Hermitian random matrices, Probab. Theory Related Fields 134 (1) (2006) 127--173. 

\bibitem{Shlyakhtenko2018} D. Shlyakhtenko.
Free probability of type-B and asymptotics of finite-rank perturbations of random matrices.
Indiana Univ. Math. J. 67 (2018), no. 2, 971--991.

\end{thebibliography}
\end{document}